\documentclass[submission]{FPSAC2023}

\usepackage{tikz-cd}

\newtheorem{theorem}{Theorem}

\newtheorem{defn}{Definition}
\newtheorem{lem}{Lemma}
\newtheorem{proposition}{Proposition}

\newtheorem{example}{Example}

\newtheorem{remark}{Remark}

\newcommand{\al}{\alpha}
\newcommand{\be}{\beta}

\newcommand{\CC}{\mathbb{C}} 
\newcommand{\PP}{\mathbb{P}} 
\newcommand{\QSym}{\operatorname{QSym}}
\newcommand{\std}{\operatorname{std}}
\newcommand{\rank}{\operatorname{rank}}
\newcommand{\Des}{\operatorname{Des}}

\newcommand{\Peak}{\operatorname{Peak}}

\makeatletter
\providecommand*{\shuffle}{%
  \mathbin{\mathpalette\shuffle@{}}%
}
\newcommand*{\shuffle@}[2]{%
  \sbox0{$#1\vcenter{}$}%
  \kern .15\ht0 
  \rlap{\vrule height .25\ht0 depth 0pt width 2.5\ht0}%
  \raise.1\ht0\hbox to 2.5\ht0{%
    \vrule height 1.75\ht0 depth -.1\ht0 width .17\ht0 %
    \hfill
    \vrule height 1.75\ht0 depth -.1\ht0 width .17\ht0 %
    \hfill
    \vrule height 1.75\ht0 depth -.1\ht0 width .17\ht0 %
  }%
  \kern .15\ht0 
}
\makeatother

\title[Extended peaks]{The algebra of extended peaks}

\author[D. Grinberg \and E.A. Vassilieva]{Darij Grinberg\thanks{\href{mailto:darijgrinberg@gmail.com}{darijgrinberg@gmail.com}}\addressmark{1}, \and Ekaterina A. Vassilieva\thanks{\href{mailto:katya@lix.polytechnique.fr}{katya@lix.polytechnique.fr}}\addressmark{2}}

\received{\today}

\abstract{Building up on our previous works regarding $q$-deformed $P$-partitions, we introduce a new family of subalgebras for the ring of quasisymmetric functions. Each of these subalgebras admits as a basis a $q$-analogue to Gessel's fundamental quasisymmetric functions where $q$ is equal to a complex root of unity. Interestingly, the basis elements are indexed by sets corresponding to an intermediary statistic between peak and descent sets of permutations that we call extended peak.}

\resume{En nous appuyant sur nos travaux précédents concernant les $P$-partitions $q$-déformées, nous introduisons une nouvelle famille de sous-algèbres pour l'anneau des fonctions quasi-symétriques. Chacune de ces sous-algèbres admet comme base un $q$-analogue aux fonctions quasisymétriques fondamentales de Gessel où $q$ est égal à une racine complexe de l'unité. Il est notable que les éléments de base sont indexés par des ensembles correspondant à une statistique intermédiaire entre les ensembles de pics et de descentes des permutations que nous appelons pic étendu.}


\usepackage[backend=bibtex]{biblatex}
\begin{document}

\maketitle
\section{Introduction}
In \cite{GriVas22}, we show that a $q$-deformation of the generating functions for $P$-partitions leads to a unified framework between classical and enriched $P$-partitions. In particular, we introduce our $q$-fundamental quasisymmetric functions that interpolate between I. Gessel's fundamental (\cite{Ges84}, $q=0$) and J. Stembridge's peak (\cite{Ste97}, $q=1$) quasisymmetric functions. When $q$ is not a root of unity, $q$-fundamentals are a basis of $\QSym$, the ring of quasisymmetric functions and are indexed by descent sets of permutations. If $q=1$, they span the subalgebra of $\QSym$ named the algebra of peaks. The relevant basis elements are those indexed by peak sets. As it turns out, using other complex roots of unity for $q$, we are able to build new intermediate subalgebras between the algebra of peaks and $\QSym$, the basis of which are $q$-fundamentals indexed by a new permutation statistic that lies between peak and descent sets. We call this statistic the extended peak set and the corresponding subalgebras of quasisymmetric functions the algebra of extended peaks. We begin with the required definitions and results from \cite{GriVas22}. Then we introduce and prove the new results regarding the algebra of extended peaks. 
\subsection{Permutation statistics}
Let $\PP$ be the set of positive integers. For $m, n \in \PP$, write $[m,n] = \{m, m+1, \dots, n\}$ and simply $[n] = \{1, 2, \dots, n\}$. We denote $S_n$ the symmetric group on $[n]$. Given $\pi \in S_n$, define its \emph{descent set} $$\Des(\pi)= \{1\leq i\leq n-1| \pi(i)>\pi(i+1)\} \subset [n-1]$$ and its \emph{peak set} $$\Peak(\pi) = \{2\leq i\leq n-1| \pi(i-1)<\pi(i)>\pi(i+1)\}.$$ The peak set of a permutation is said to be \emph{peak-lacunar}, i.e. it neither contains $1$ nor contains two consecutive integers.
\subsection{Enriched $P$-partitions and $q$-deformed generating functions}
\label{sec : poset}
We recall the main definitions regarding weighted posets, enriched $P$-partitions and their $q$-deformed generating functions. The reader is referred to \cite{Ges84, GriVas21, GriVas22, Sta01, Ste97} for more details.   
\begin{defn}[Labelled weighted poset, \cite{GriVas21}]
A \emph{labelled weighted poset} is a triple $P = ([n],<_P,\epsilon)$ where $([n], <_P)$ is a \emph{labelled poset}, i.e., an arbitrary partial order $ <_P$ on the set $[n]$ and $\epsilon : [n]\longrightarrow \PP$ is a map (called the \emph{weight function}).
\end{defn}
\noindent Each node of a labelled weighted poset is marked with its label and weight (Figure \ref{fig : poset}).
\begin{figure}[htbp]
\begin{center}
\begin{tikzcd}[row sep = small]
                                              &  & {2,\ \epsilon(2) = 5} &  &                                  &  &                       \\
{3,\ \epsilon(3) = 2} \arrow[rru,very thick] &  &                                   &  & {1,\ \epsilon(1) = 1} \arrow[llu,very thick] \arrow[rr,very thick] &  & {4,\ \epsilon(4) = 2} \\
                                              &  & {5,\ \epsilon(5)=2} \arrow[llu,very thick]  \arrow[rru,very thick]   &  &                                  &  &                      
\end{tikzcd}
\end{center}
\caption{A $5$-vertex labelled weighted poset. Arrows show the covering relations.}
 \label{fig : poset}
 \end{figure}
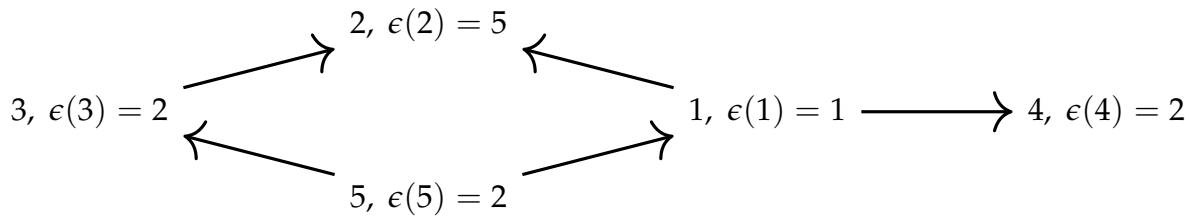
 \begin{defn}[Enriched $P$-partition, \cite{Ste97}]\label{def : enriched}
Let $\PP^{\pm}$ be the set of positive and negative integers totally ordered by $-1<1<-2<2<-3<3<\dots$. We embed $\PP$ into $\PP^{\pm}$ and let $-\PP \subseteq \PP^{\pm}$ be the set of all $-n$ for $n \in \PP$. Given a labelled weighted poset $P = ([n],<_P,\epsilon)$, an \emph{enriched $P$-partition} is a map $f: [n]\longrightarrow \PP^{\pm} $ that satisfies the two following conditions:
\begin{itemize}
\item[(i)] If $i <_P j$ and $i < j$, then $f(i) < f(j)$ or $f(i) = f(j) \in \PP$.
\item[(ii)] If $i <_P j$ and $i>j$, then $f(i) < f(j)$ or $f(i) = f(j) \in -\PP$.
\end{itemize}
We denote $\mathcal{L}_{\PP^{\pm}}(P)$ be the set of enriched $P$-partitions.
\end{defn} 
\begin{defn}[$q$-Deformed generating function, \cite{GriVas22}]
Consider the set of indeterminates $X = \left\{x_1,x_2,x_3,\ldots\right\}$, the ring $\CC \left[\left[ X \right]\right]$ of formal power series on $X$ where $\CC$ is the set of complex numbers, and let $q \in \CC$ be an additional parameter. Given a labelled weighted poset $([n], <_P, \epsilon)$, define its generating function $\Gamma^{(q)}([n], <_P, \epsilon) \in \CC \left[\left[ X \right]\right]$ as
\begin{equation*}
\label{eq : weightGamma}
\Gamma^{(q)}([n], <_P, \epsilon) = \sum_{f\in\mathcal{L}_{\PP^\pm}([n],<_{P}, \epsilon)} \prod_{1\leq i\leq n}q^{[f(i)<0]}x_{|f(i)|}^{\epsilon(i)},
\end{equation*}
where $[f(i)<0] = 1$ if $f(i)<0$ and $0$ otherwise.
\end{defn}

\subsection{Enriched $q$-monomial and $q$-fundamental quasisymmetric functions}
We state without proofs the required definitions and propositions from \cite{GriVas22}. The main building block of this previous work is the $q$-deformed generating function for enriched $P$-partitions on labelled weighted chains that we call \emph{universal quasisymmetric functions}. 
\begin{defn}[Universal quasisymmetric functions]
Given a \emph{composition}, i.e. a sequence of positive integers $\alpha = (\alpha_1, \alpha_2, \ldots, \alpha_n)$ with $n$ entries, and a permutation $\pi=\pi_1\dots\pi_n$ of $S_n$, we let $P_{\pi,\alpha} = ([n],<_\pi,\alpha)$ be the labelled weighted poset on the set $[n]$, where the order relation $<_\pi$ is such that $\pi_i <_\pi \pi_j$ if and only if $i < j$ and where $\alpha$ is the weight function sending the vertex labelled $\pi_i$ to $\alpha_i$ (see Figure \ref{fig : monomial}).
Define the \emph{$q$-universal quasisymmetric function}
\begin{equation*}
U^{(q)}_{\pi,\alpha} = \Gamma^{(q)}([n],<_\pi, \alpha).
\end{equation*}
\begin{figure}[htbp]
\begin{center}
\begin{tikzcd}[row sep = small]
{\pi_1,\ \alpha_1} \arrow[r, very thick] &
{\pi_2,\ \alpha_2} \arrow[r, very thick] &
\cdots\cdots\cdots \arrow[r, very thick] &
{\pi_n,\ \alpha_n}
\end{tikzcd}
\end{center}
\caption{The labelled weighted poset $P_{\pi,\alpha}$.}
\label{fig : monomial}
\end{figure}
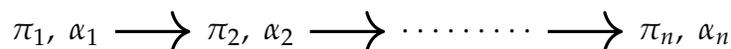
\end{defn}
\noindent Universal quasisymmetric functions belong to the subalgebra of $\CC \left[\left[ X \right]\right]$ called the ring of \emph{quasisymmetric functions ($\QSym$)}, i.e. for any strictly increasing sequence of indices $i_1 < i_2 <\cdots< i_p$ the coefficient of $x_1^{k_1}x_2^{k_2}\cdots x_p^{k_p}$ is equal to the coefficient of $x_{i_1}^{k_1}x_{i_2}^{k_2}\cdots x_{i_p}^{k_p}$. They are directly connected to classical bases of $\QSym$ as $L_{\pi}= U^{(0)}_{\pi,[1^n]}$ (resp. $K_{\pi}= U^{(1)}_{\pi,[1^n]}$) is the Gessel's  \emph{fundamental} (\cite{Ges84})  (resp. Stembridge's  \emph{peak}, \cite{Ste97}) quasisymmetric function indexed by $\pi$. Moreover, if we let $id_{n} = 1~2\dots n $ and $\overline{id_{n}} = n~n-1\dots 1$, then $U^{0}_{\overline{id_{n}},\alpha} = M_\al$ is the \emph{monomial}  (\cite{Ges84}), $U^{(0)}_{id_{n},\alpha} = E_\al$ the \emph{essential}  (\cite{Hof15})  and $U^{(1)}_{id_{n},\alpha} = \eta_\al$ the \emph{enriched monomial}  (\cite{Hsi07, GriVas21}) quasisymmetric functions indexed by $\al$. Moreover universal quasisymmetric functions satisfy the explicit expression
\begin{equation}
U_{\pi,\alpha}^{(q)}=\sum_{\substack{i_{1}\leq i_{2}\leq
\dots\leq i_{n};\\j\in\Peak(\pi)\Rightarrow 
i_{j-1}<i_{j+1}  }}q^{|\{j\in\Des(\pi
)|i_{j}=i_{j+1}\}|}(q+1)^{|\{i_{1},i_{2},\dots, i_{n}\}|}x_{i_{1}}^{\alpha_{1}%
}x_{i_{2}}^{\alpha_{2}}\dots x_{i_{n}}^{\alpha_{n}}.
\label{eq : Uqq}
\end{equation}
As universal quasisymmetric functions are the generating functions of some $P$-partitions on labelled chains, they admit the closed form product rule of Proposition \ref{prop : UUU}. 
\begin{proposition}[Product rule]
\label{prop : UUU}
Let $q \in \CC$, let $\pi$ and $\sigma$ be two permutations in $S_n$ and $S_m$, and let $\alpha = (\al_1,\dots,\al_n)$ and $\beta = (\be_1,\dots,\be_m)$ be two compositions with $n$ and $m$ entries.
The product of two $q$-universal quasisymmetric functions is given by
\begin{equation}
\label{eq : UU}
U^{(q)}_{\pi,\alpha}U^{(q)}_{\sigma,\beta}=\sum_{(\tau,\gamma)\in(\pi,\alpha)\shuffle(\sigma,\beta)}U^{(q)}_{\tau,\gamma} .
\end{equation}
Here $(\pi,\alpha) \shuffle (\sigma,\beta)$ denote the \emph{coshuffle} of $(\pi,\alpha)$ and $(\sigma,\beta)$, i.e. the set of pairs $(\tau,\gamma)$ where $\tau \in S_{n+m}$ is a shuffle of $\pi$ and $n+\sigma = (n+\sigma_1, \dots, n+\sigma_m)$, and $\gamma$ is a composition with $n+m$ entries, obtained by shuffling the entries of $\alpha$ and $\beta$ using the same shuffle used to build $\tau$. 
\end{proposition}
\begin{remark}[Coproduct]
It's easy to notice that universal quasisymmetric functions admit a coproduct $\Delta:\operatorname*{QSym}\rightarrow
\operatorname*{QSym}\otimes\operatorname*{QSym}$ of the Hopf algebra
$\operatorname*{QSym}$ (see \cite[\S 5.1]{GriRei20}). Let $n \in \PP$, $\pi \in S_n$ and $\alpha$ be a composition with $n$ entries.
\begin{equation*}
\Delta(U^{(q)}_{\pi, \alpha})
= \sum_{i=0}^n U^{(q)}_{\std(\pi_1\pi_2\dots \pi_i), (\al_1,\al_2,\dots,\al_i)} \otimes
U^{(q)}_{\std(\pi_{i+1}\pi_{i+2}\dots \pi_n), (\al_{i+1},\al_{i+2},\dots,\al_n)}.
\end{equation*}
Here, if $\gamma$ is a sequence of non repeating integers, $\std(\gamma)$ is the permutation whose values are in the same relative order as the entries of $\gamma$.
\end{remark}
Our work relies on two significant specialisations of universal quasisymmetric functions called enriched $q$-monomial and $q$-fundamental quasisymmetric functions.  
\begin{defn}[Enriched $q$-monomial quasisymmetric functions]
\label{def : EQ}
Let $q \in \CC$ and $\al$ be a composition with $n$ entries. The \emph{enriched $q$-monomial} indexed by $\al$ is defined as
\begin{equation*}
\label{eq : EU}
\eta^{(q)}_\al = U^{(q)}_{id_n,\alpha}= \sum_{i_1\leq i_2 \leq \dots \leq i_n}(q+1)^{|\{i_{1},i_{2},\dots, i_{n}\}|}x_{i_{1}}^{\alpha_{1}}x_{i_{2}}^{\alpha_{2}}\dots x_{i_{n}}^{\alpha_{n}}.
\end{equation*}
As compositions $\al = (\al_1,\dots,\al_n)$ such that $\al_1 + \dots + \al_n = s$ are in bijection with subsets of $[s-1]$, we also use the equivalent expression for $I \subseteq [s-1]$.
\begin{equation*}
\eta^{(q)}_{s, I} =\sum_{\substack{i_1\leq\dots\leq i_s\\ j \in I \Rightarrow i_j=i_{j+1}}}(q+1)^{|\{i_1,\dots,i_s\}|}x_{i_1}\dots x_{i_s}.
\end{equation*}
\end{defn}
As an immediate consequence of Definition \ref{def : EQ}, $\eta^{(0)}_\al$ (resp. $\eta^{(1)}_\al$)  is Hoffman's essential (\cite{Hof15})  (resp. enriched monomial, \cite{GriVas21}) quasisymmetric function indexed by the composition $\alpha$. Except the degenerate case $q=-1$, $q$-enriched monomials are a basis of $\QSym$. 
\begin{proposition}
Let $q \in \CC \setminus \{-1\}$. The family $\left(\eta^{(q)}_{s,I} \right )_{s\geq0, I\subseteq[s-1]}$ is a basis of $\QSym$.
\end{proposition}
\begin{defn}[$q$-Fundamental quasisymmetric functions]
Let $\pi$ be a permutation in $S_n$ and $q \in \CC$. Define the \emph{$q$-fundamental quasisymmetric function} indexed by $\pi$ as
\begin{equation*}
L_{\pi}^{(q)} = U^{(q)}_{\pi, [1^{n}]}.
\end{equation*}
According to Equation (\ref{eq : Uqq}), $L_{\pi}^{(q)}$ depends only on the descent set of $\pi$. As a result, $q$-fundamentals are naturally indexed by sets and we may denote them  $(L_{n, I}^{(q)})_{I \subseteq [n-1]}$. 
\end{defn}
The specialisations of $L_{\pi}^{(q)}$ to $q=0$ and $q=1$ are respectively the Gessel's fundamental \cite{Ges84} and Stembridge's peak \cite{Ste97} quasisymmetric functions indexed by permutation $\pi$. The expression of $q$-fundamentals in the basis of enriched $q$-monomials is of particular importance. 
\begin{proposition} Let $I \subseteq[n-1]$ and $q \in \CC$. Let also $\Peak(I) = I\setminus (I-1)\setminus \{1\}$. The $q$-fundamental quasisymmetric functions may be expressed in the enriched $q$-monomial basis as
\begin{equation}
\label{eq : LEq}
L_{n, I}^{(q)}
= \sum_{\substack{J \subseteq I\\K \subseteq \Peak(I)\\J \cap K = \emptyset}}
(-q)^{|K|}(q-1)^{|J|}\eta^{(q)}_{n, J\cup (K-1) \cup K}.
\end{equation}
\end{proposition}
Unlike $q$-monomials, $q$-fundamentals are not always a basis of $\QSym$.
\begin{proposition}\label{thm : basis} $(L_{n, I}^{(q)})_{n\geq0, I\subseteq[n-1]}$ is a basis of $\QSym$ if and only if $q \in \CC$ is not a root of unity.
\end{proposition}
\section{Extended peaks}
According to Proposition \ref{thm : basis}, $q$-fundamental quasisymmetric functions are very similar to classical ones when $q$ is not a root of unity and are naturally indexed by descent sets. When $q=1$, they reduce to peak quasisymmetric functions, are indexed by peak sets and span a very significant subalgebra of $\QSym$. Understanding what subalgebra they span and how to index them when $q$ is another root of unity appears as a natural question. We begin with the introduction of the appropriate sets. 
\subsection{Extended peak sets}
We use the following subsets with constraints on consecutive elements.
\begin{defn}[Extended peak set]
Let $n$ and $p$ be two positive integers. We say that $I\subseteq [n-1]$ is a \emph{$p$-extended peak set} if $I \cup \{0\}$ doesn't contain more than $p$ consecutive elements (as a result, $[1, p] \nsubseteq I$). 
We write $I \subseteq_p [n-1]$ for this statement.
\end{defn}
\begin{example}
Set $n=9$. One has $\{4,8\} \subseteq_1 [8]$, $\{1, 4,5, 8\} \subseteq_2 [8]$, $\{1,2,4,5,6,8\} \subseteq_3 [8]$. However $\{1,2, 4,5, 8\} \nsubseteq_2 [8]$ as the subsequence $[1,2]$ containing $1$ has size $2 > p-1 = 1$.
\end{example}
\begin{remark}[Permutation statistics] Extended peak sets look like an intermediary statistic between peak and descent sets. Any peak set is a $1$-extended peak set and any descent set on permutations of $n$ elements is a $p$-extended peak set for $p\geq n$. Moreover, given a permutation $\pi$ in $S_n$ and an integer $p \geq 1$ one may define ${\Peak}_p(\pi)$ as 
\begin{align*}
{\Peak}_p(\pi) = \{i \in \Des(\pi) | i \leq p - 1 \mbox{ or } \exists \; 1 \leq j \leq p \mbox{ such that } i-j \notin \Des(\pi)\}.
\end{align*}
On the one hand one has always $\Peak(\pi) = {\Peak}_1(\pi) \subseteq {\Peak}_p(\pi) \subseteq \Des(\pi)$. On the other hand, ${\Peak}_p(\pi) \subseteq_p [n-1]$. For instance let $\pi = 54163287$. We have ${\Peak}_1(\pi) = \Peak(\pi) = \{4,7\} \subseteq {\Peak}_2(\pi)  = \{1,4,5,7\} \subseteq {\Peak}_3(\pi)  = \{1,2,4,5,7\} = \Des(\pi)$.
\end{remark}
We count the number of extended peak sets.
\begin{proposition}
Let $n,p \in \PP$. Denote $s^{(p)}_n$ be the number of $p$-extended peak sets on $n$ elements. Extend the definition with $s^{(p)}_0 = 0$ for all positive $p$. 
One has
\begin{equation}
s^{(p)}_n = 
\begin{cases}
2^{n-1}\; \mbox{ if } \; n \leq p\\
\sum_{k=0}^{p} s^{(p)}_{n-k-1} = \sum_{k=1}^{p+1} s^{(p)}_{n-k}\; \mbox{ if } \; n > p\\
\end{cases}
\end{equation}
\end{proposition}
\begin{proof}
The result is immediate for $n \leq p$. For $n>p$, there is a  bijection between $p$-extended peak sets on $n$ elements and the union of $p$-extended peak sets on $n-1-k$ elements for $k \in [0,p]$. Indeed, given a set $I \subseteq_p [n-1]$, define $k \leq p$ as the integer such that $[n-k, n-1]$ is the maximum sequence of consecutive integers in $I$ containing $n-1$. If $n-1 \notin I$ we define $k=0$ and assume $[n, n-1] = \emptyset$. We map $I$ to the unique element $J \subseteq_p [n-k-2]$ such that $I = J \cup [n-k, n-1]$. This mapping is clearly one-to-one and the result follows. 
\end{proof}
\subsection{Extended peak quasisymmetric functions}
\noindent We proceed with the definition the relevant subfamilies of $q$-fundamentals.
\begin{defn}[Extended-peak quasisymmetric functions]
\label{def.extended_peak_qsym}
Let $n, p \in \PP$ and denote $\rho_p$ the root of unity $\rho_p = e^{-i\pi (p-1)/(p+1)}$. We have $\rho_1 = 1$, $\rho_2 = e^{-i\pi/3}$, $\rho_4 = e^{-i\pi/2},\dots$. Note that $(-\rho_p)$ is a primitive $p+1$-th root of unity, i.e., $(-\rho_p)^{p+1}=1$ but $(-\rho_p)^j \neq 1$ for $1\leq j < p+1$. Given a subset $I\subseteq [n-1]$ we define the $p$-extended peak quasisymmetric function indexed by $I$
\begin{equation}
L^p_{n, I} = L^{(\rho_p)}_{n, I}. 
\end{equation}
Denote $\mathcal{P}^p \subseteq \QSym$ the subalgebra of $\QSym$ spanned by $(L^p_{n, I})_{n\geq0, I\subseteq[n-1]}$ and $\mathcal{P}^p_n \subseteq \QSym_n$ its subspace composed of quasisymmetric functions of degree $n$ (i.e the vector space spanned by $(L^p_{n, I})_{I\subseteq[n-1]}$). We call $\mathcal{P}^p$ the \emph{algebra of $p$-extended peaks}. 
\end{defn}
Definition \ref{def.extended_peak_qsym} gives extended peak functions over all subsets. However, we know from Proposition \ref{thm : basis} that they do not span $\QSym$. As a result, for all $p \in \PP$, the family $(L^p_{n, I})_{n\geq0, I\subseteq[n-1]}$ is not linearly independent  and some indices are redundant. We characterise these set indices. First, for $n, p \in \PP$, if a set $I$ is not a $p$-extended peak set, then $L^p_{n, I}$ may be expressed in terms of other $p$-extended peak quasisymmetric functions.
\begin{theorem}[Extended peak functions over sets that are not $p$-extended peaks] 
\label{thm.non_valid_sets}
Let $n, p \in \PP$ with $n \geq p+1$, $i$ be an integer such that $0 \leq i \leq n - 1 -p$ and $J \subseteq [n-1]$ be a subset that satisfies $[i+1,i+p+1] \cap J = \emptyset$ and $i \in J \cup \{0\}$. Then, the set $ [i+1,i+p] \cup J \nsubseteq_p [n-1]$ as it contains either a sequence of $p+1$ consecutive elements or the sequence $[1,p]$. Notice further that any set that is not a $p$-extended peak set may be written as such. We have the following equality.
\begin{equation*}
\sum_{I \subseteq [i+1, i+p]}(-1)^{|I|}L^p_{n, I \cup J} = 0.
\end{equation*}
\end{theorem}
Secondly, we can compute explicitly the dimension of $\mathcal{P}^p_n$ for $n, p \in \PP$.
\begin{theorem}[Subspaces dimension]
\label{thm.dimension}
Let $n, p \in \PP$ be two positive integers. The dimension of $\mathcal{P}^p_n$ is equal to $s^{(p)}_n$, the number of $p$-extended peak sets on $n$ elements.
\begin{equation*}
\dim \mathcal{P}^p_n = s^{(p)}_n
\end{equation*}
\end{theorem}
We postpone the proofs of Theorems \ref{thm.non_valid_sets} and \ref{thm.dimension} respectively to Sections \ref{section.non_valid_sets} and \ref{section.dimension}. Combining them we characterise the subalgebra $\mathcal{P}^p$. 
\begin{theorem}[Basis for the algebra of extended peaks]
\label{thm.algebra_extended_peaks}
Let $p \in \PP$. The family $(L_{n, I}^{p})_{n\geq0, I\subseteq_p[n-1]}$ is a basis of the subalgebra $\mathcal{P}^p$ of $\QSym$.
\end{theorem}

\begin{proof}
Fix $p \in \PP$. As $p$-extended peak quasisymmetric functions are special cases of $q$-fundamentals, the stability by multiplication is actually a direct consequence of Equation (\ref{eq : UU}). Then Theorem \ref{thm.non_valid_sets} shows that only $p$-extended peak quasisymmetric functions indexed by $p$-extended peak sets may be linearly independent. Finally, Theorem \ref{thm.dimension} shows that for all $n \in \PP$ the dimension of the finite vector space containing homogenous quasisymmetric functions of degree $n$ is exactly the number of $p$-extended peak sets on $n$ elements. 
\end{proof}

\section{Proofs of Theorems \ref{thm.non_valid_sets} and \ref{thm.dimension}}
\subsection{Extended peak functions indexed by generic sets}
\label{section.non_valid_sets}
In order to show Theorem \ref{thm.non_valid_sets} compute
\begin{equation*}
\sum_{I \subseteq [i+1, i+p]}(-1)^{|I|}L^{(q)}_{I \cup J}
\end{equation*}
using Equation (\ref{eq : LEq}) for integers $i, n, p$ and set $J$ satisfying the conditions of the theorem. Note that for $I \subseteq [i+1, i+p]$, $i+1 \notin \Peak(I \cup J)$ (either $i+1 = 1$ or $i \in J$) and that $\Peak(I \cup J) \cap J = \Peak(J)$ irrelevant of the choice of $I$ (as $i+p+1 \notin J$). As a result, we can decompose in Equation (\ref{eq : LEq}) any subset or peak lacunar subset (i.e. $1$-extended peak subset) of $I\cup J$ as a (peak lacunar) subset of $I$ and a (peak lacunar) subset of $J$. Namely,
\begin{align*}
&\sum_{I \subseteq [i+1, i+p]}(-1)^{|I|}L^{(q)}_{I \cup J} \\
&=  \sum_{\substack{U' \subseteq J\\ V' \subseteq \Peak(J) \\ U'\cap V' = \emptyset}}\!\!\!\!\!\!\!(-q)^{|V'|}(q-1)^{|U'|}\!\!\!\!\!\sum_{I \subseteq [i+1, i+p]}(-1)^{|I|}\!\!\!\!\!\!\!\!\!\!\!\!\sum_{\substack{U \subseteq I\\ V \subseteq \Peak(I)\setminus \{i+1\} \\ U\cap V = \emptyset}}\!\!\!\!\!\!\!\!\!\!\!(-q)^{|V|}(q-1)^{|U|}\eta^{(q)}_{U' \cup V'-1 \cup V' \cup U \cup V-1 \cup V}.
\end{align*}
Next, invert summation indices in the last sums.
\begin{align*}
&\sum_{I \subseteq [i+1, i+p]}(-1)^{|I|}L^{(q)}_{I \cup J} \\
&=  \sum_{\substack{U' \subseteq J\\ V' \subseteq \Peak(J) \\ U'\cap V' = \emptyset}}\!\!\!\!\!\!\!(-q)^{|V'|}(q-1)^{|U'|}\!\!\!\!\!\sum_{\substack{U \subseteq [i, i+p]\\ V \subseteq_1 [i+1, i+p] \\ U\cap V = \emptyset\\ U\cap V-1 = \emptyset}}\!\!\!\!\!\!\!\!(-q)^{|V|}(q-1)^{|U|}\eta^{(q)}_{U' \cup V'-1 \cup V' \cup U \cup V-1 \cup V}\!\!\!\!\!\!\!\!\!\!\!\!\!\!\!\!\!
\sum_{U \cup V \subseteq I \subseteq [i+1, i+p] \setminus V-1}\!\!\!\!\!\!\!\!\!\!\!\!\!\!\!\!\!(-1)^{|I|}.
\end{align*}
The last sum is obviously $0$ except when $U \cup V = [i+1, i+p] \setminus V-1$. As a result, 
\begin{align*}
&\sum_{I \subseteq [i+1, i+p]}(-1)^{|I|}L^{(q)}_{I \cup J} \\
&=  \sum_{\substack{U' \subseteq J\\ V' \subseteq \Peak(J) \\ U'\cap V' = \emptyset}}\!\!\!\!\!\!\!(-q)^{|V'|}(q-1)^{|U'|} \eta^{(q)}_{U' \cup V'-1 \cup V' \cup [i,i+p]}\sum_{V \subseteq_1 [i+1, i+p]}\!\!\!\!\!\!\!\!(-q)^{|V|}(q-1)^{p - 2|V|}(-1)^{p - |V|}.
\end{align*}
The summands in the sum over all $1$-extended peak sets $V \subseteq_1 [i+1, i+p]$ depend only on the cardinality of $V$. It easy to show (left to the reader) that
\begin{equation*}
|\{V\subseteq_1 [i+1, i+p], |V| = v\}| = \binom{p - v}{v}.
\end{equation*}
Subsequently,
\begin{align*}
&\sum_{I \subseteq [i+1, i+p]}(-1)^{|I|}L^{(q)}_{I \cup J} \\
&= (-1)^{p}\!\!\!\!\!\!\!\!\! \sum_{\substack{U' \subseteq J\\ V' \subseteq \Peak(J) \\ U'\cap V' = \emptyset}}\!\!\!\!\!\!\!(-q)^{|V'|}(q-1)^{|U'|} \eta^{(q)}_{U' \cup V'-1 \cup V' \cup [i+1,i+p]}\sum_{v = 0}^p(-1)^v\binom{p-v}{v}(-q)^{v}(q-1)^{p - 2v}.
\end{align*}
We use the following lemma.
\begin{lem}[\cite{Sur04}]
\label{lem.binom}
Let $n \in \PP$ and $x,y \in \CC$, one has
\begin{equation*}
\sum_{k = 0}^n(-1)^k\binom{n-k}{k}(xy)^k(x+y)^{n-2k} = \sum_{j=0}^n x^{n-j}y^j.
\end{equation*}
\end{lem}
\noindent Denote for $n \in \PP, c \in \CC$, $[n]_c = (1-c^n)/(1-c)$. As a direct consequence of Lemma \ref{lem.binom},
\begin{align*}
\sum_{I \subseteq [i+1, i+p]}(-1)^{|I|}L^{(q)}_{I \cup J} &=  \sum_{\substack{U' \subseteq J\\ V' \subseteq \Peak(J) \\ U'\cap V' = \emptyset}}\!\!\!\!\!\!\!(-q)^{|V'|}(q-1)^{|U'|} \eta^{(q)}_{U' \cup V'-1 \cup V' \cup [i+1,i+p]}\sum_{t= 0}^{p}(-q)^{p-t},\\
&=[p+1]_{-q}\sum_{\substack{U' \subseteq J\\ V' \subseteq \Peak(J) \\ U'\cap V' = \emptyset}}\!\!\!\!\!\!\!(-q)^{|V'|}(q-1)^{|U'|} \eta^{(q)}_{U' \cup V'-1 \cup V' \cup [i+1,i+p]}.
\end{align*}
End the proof with 
\begin{equation*}
[p+1]_{-\rho_p} = \frac{1 - (-\rho_p)^{p+1}}{1+\rho_p} = 0.
\end{equation*}
\subsection{Finite subspaces dimension}
\label{section.dimension}
For $n \in \PP$ and $q \in \CC$, denote $B_n^{(q)}$ the transition matrix between $(L_{n, I}^{(q)})_{I \subseteq [n-1]}$ and $(\eta^{(q)}_{n, J})_{J \subseteq [n-1]}$ with coefficients given by Equation (\ref{eq : LEq}). Columns and rows are indexed by subsets $I$ of $[n-1]$ sorted in reverse lexicographic order. A subset $I$ is before subset $J$ iff the word obtained by writing the elements of $I$ in decreasing order is before the word obtained from $J$ for the lexicographic order. The column indexed by the subset $I$ corresponds to $L_{n, I}^{(q)}$ and the row indexed by $J$ to $\eta^{(q)}_{n, J}$ (as a direct consequence $B_n^{(q)}$ is the transpose of the similar matrix defined in \cite{GriVas22}). For $n=0$, assume $B_0^{(q)}$ to be the empty matrix. 
\begin{example} 
For $n=4$, the transition matrix $B_{4}^{(q)}$ between $(L_{I}^{(q)})_{I \subseteq [3]}$ and $(\eta^{(q)}_{J})_{J \subseteq [3]}$ is given by
\begin{equation*}
B_{4}^{(q)} = 
\begin{array}{c|cccccccc}
 & \emptyset & \{1\} & \{2\} &  \{2, 1\} & \{3\} &  \{3, 1\}& \{3, 2\} &  \{3, 2, 1\}\\
 \hline
 \emptyset & 1 & 1 & 1 & 1 & 1 & 1 & 1 & 1\\
\{1\} & 0 & q-1 & 0 & q-1 & 0 & q-1 & 0 & q-1\\
\{2\} & 0 & 0 & q-1 & q-1 & 0 & 0 & q-1 & q-1\\
\{2, 1\} & 0 & 0 & -q & (q-1)^2 & 0 & 0 & -q & (q-1)^2\\
\{3\} & 0 & 0 & 0 & 0 & q-1 & q-1 & q-1 & q-1\\
\{3, 1\} & 0 & 0 & 0 & 0 & 0 & (q-1)^2 & 0 & (q-1)^2 \\
\{3, 2\} & 0 & 0 & 0 & 0 & -q & -q & (q-1)^2 & (q-1)^2\\
\{3,2,1\} & 0 & 0 & 0 & 0 & 0 & -q(q-1) & -q(q-1) & (q-1)^3\\
\end{array}
\end{equation*}
\end{example}
Our goal is to compute the dimension of the kernel of $B_n^{(q)}$ to get the dimension of the vector subspace $\mathcal{P}^p_n$ as $$\dim \mathcal{P}^p_n = \rank(B_n^{(q)}) = 2^{n-1} - \dim \ker B_n^{(q)}.$$
We show the following proposition.
\begin{proposition}
\label{prop.kernel}
Let $n, p \in \PP$ be two positive integers. We have
\begin{equation*}
\dim \ker B_n^{(\rho_p)} = \begin{cases}\sum_{k =1}^{p+1} \dim \ker B_{n-k}^{(\rho_p)} + [n > p+1]2^{n-p-2}\;\mbox{ for } n > p,\\
0\;\mbox{ for } n \leq p. \end{cases}
\end{equation*}
\end{proposition}
\begin{proof}
The second case is a direct consequence of the fact that the matrix $B_n^{(\rho_p)}$ is invertible for $n\leq p$ (see \cite{GriVas22}).
To show the general recurrence, assume that $n>p$. As in \cite{GriVas22}, notice that
the matrix $B_n^{(q)}$ is block upper triangular. For each $k \in [n]$, let $A_k^{(q)}$ denote the transition matrix from $(L_{n, I}^{(q)})_{I \subseteq [n-1],~  \max(I) = k-1}$ to $(\eta^{(q)}_{n, J})_{J \subseteq [n-1],~\max(J) = k-1}$ (where $\max\varnothing := 0$); this actually does not depend on $n$. Note that $A_k^{(q)}$ is a $2^{k-2} \times 2^{k-2}$-matrix if $k \geq 2$, whereas $A_1^{(q)}$ is a $1 \times 1$-matrix.
We have
\begin{equation*}
B_{n}^{(q)} = 
\begin{pmatrix}
 A_1^{(q)}& *& *& \hdots& *\\
 0&A_2^{(q)}& *& \hdots& *\\
 0&0&A_3^{(q)}&\hdots& *\\
 0&0&0&\ddots&*\\
 0&0&0&0&A_n^{(q)}\\
 \end{pmatrix}.
\end{equation*}
We have the following lemma:
\begin{lem}
\label{lem : rec}
The matrices $\left(B_n^{(q)}\right)_n$ and $\left (A_n^{(q)}\right)_n$ satisfy the following recurrence relations (for $n \geq 1$ and $n \geq 2$, respectively):
\begin{equation*}
B_{n}^{(q)} = 
\begin{pmatrix}
 B_{n-1}^{(q)}&B_{n-1}^{(q)}\\
0 &A_n^{(q)}\\
 \end{pmatrix},
 \qquad
A_{n}^{(q)} = 
\begin{pmatrix}
 (q-1)B_{n-2}^{(q)}& (q-1)B_{n-2}^{(q)}\\
-qB_{n-2}^{(q)}&(q-1)A_{n-1}^{(q)}\\
 \end{pmatrix}.
\end{equation*}
\end{lem}
\noindent Consider for $n \geq 2$ and coefficients $\alpha, \beta \in \CC$ the kernel of the matrix $\alpha A_{n+1}^{(q)} + \beta B_n^{(q)}$.
Let $X$ be a $2^{n-1}$ vector and denote $X^1$ and $X^2$ the two vectors of size $2^{n-2}$ such that
$
X = \begin{pmatrix}
X^1\\
X^2\\
\end{pmatrix}.
$
We compute
\begin{align}
\label{equation.kernel_recurrence}
\left(\alpha A_{n+1}^{(q)} + \beta B_n^{(q)}\right)X = 0 \Leftrightarrow \begin{cases}((q-1)\al+\beta)B_{n-1}^{(q)}(X^1+X^2) = 0\\ -q\al B_{n-1}^{(q)}X^1 + ((q-1)\al+\beta)A_{n}^{(q)}X^2 = 0 \end{cases}
\end{align}
Two cases arise from the previous equation. Either $((q-1)\al+\beta) \neq 0$ and
\begin{align}
\label{equation.kernel_recurrence.case_1}
\left(\alpha A_{n+1}^{(q)} + \beta B_n^{(q)}\right)X = 0 \Leftrightarrow \begin{cases}B_{n-1}^{(q)}(X^1+X^2) = 0\\ \left(((q-1)\al+\beta)A_{n}^{(q)} + q\al B_{n-1}^{(q)}\right)X^2 = 0 \end{cases}
\end{align}
or $((q-1)\al+\beta) = 0$ and 
\begin{align}
\label{equation.kernel_recurrence.case_2}
\left(\alpha A_{n+1}^{(q)} + \beta B_n^{(q)}\right)X = 0 \Leftrightarrow  q\al B_{n-1}^{(q)}X^1 = 0
\end{align}
Consider the sequence of coefficients
\begin{align*}
\left(\alpha_0^{(q)}\;\;\;\; \beta_0^{(q)}\right) &= (0\;\;\;\; 1)\\
\left(\alpha_{n+1}^{(q)}\;\;\;\; \beta_{n+1}^{(q)}\right) &= \left(\alpha_{n}^{(q)}\;\;\;\; \beta_{n}^{(q)}\right)\begin{pmatrix}
q-1&q\\
1&0\\
\end{pmatrix}, \mbox{ for } n \geq 0.
\end{align*}
Solving the recurrence we have for integer $n \geq 1$
\begin{align*}
\left(\alpha_{n}^{(q)}\;\;\;\; \beta_{n}^{(q)}\right) &= (0\;\;\;\; 1){\begin{pmatrix}
q-1&q\\
1&0\\
\end{pmatrix}}^n\\
&= (-1)^n(0\;\;\;\; 1){\begin{pmatrix}
[n+1]_{-q}&-q[n]_{-q}\\
-[n]_{-q}&q[n-1]_{-q}\\
\end{pmatrix}}\\
&=(-1)^n(-[n]_{-q}\;\;\;\; q[n-1]_{-q}),
\end{align*}
where for integer $i$ and complex number $c$ recall that $[i]_c = (1 - c^i)/(1-c)$. Finally notice that for integer $n \geq 1$
\begin{align*}
\alpha_{n}^{(q)} (q-1) + \beta_{n}^{(q)} &= (-1)^n\left([n]_{-q} - q(-q)^{n-1}\right)\\
&=(-1)^n[n+1]_{-q}.
\end{align*}
Get back to the case $q=\rho_p$ for some positive integer $p \in \PP$. As a direct consequence of the defintion of $\rho_p$ in Definition \ref{def.extended_peak_qsym}, we have
\begin{align*}
&[p+1]_{-\rho_p} = 0\\
&[n]_{-\rho_p} \neq 0,\; 1\leq n \leq p
\end{align*}
As a result, if $q=\rho_p$, one may iterate the recurrence in Equation (\ref{equation.kernel_recurrence}) $p$ times with the case of Equation (\ref{equation.kernel_recurrence.case_1}) and one more time to go to the case of Equation (\ref{equation.kernel_recurrence.case_2}). Recall that $\beta_{p+1}^{(\rho_p)} = \rho_p[p]_{-\rho_p} \neq 0$ to conclude the proof.
\end{proof}
Noticing that for $n > p+1$, $2^{n-1} = 2^{n-2} + 2^{n-3} + \dots + 2^{n-p-1} + 2\cdot 2^{n-p-2}$ we can deduce the rank of $B_n^{(\rho_p)}$ using Proposition \ref{prop.kernel}. We get
\begin{equation}
\rank \left (B_n^{(\rho_p)}\right) = 2^{n-1} - \dim \ker B_n^{(\rho_p)} = \begin{cases}\sum_{k =1}^{p+1} \rank\left(B_{n-k}^{(\rho_p)}\right)\;\mbox{ for } n > p,\\
2^{n-1}\;\mbox{ for } 1 \leq n \leq p.\end{cases}
\end{equation}
We conclude that the sequence of subspace dimensions $(\mathcal{P}^p_n)_n$ follows the same recurrence with the same initial conditions as the sequence of the numbers of $p$-extended peak sets $(s_n^p)_n$. Theorem \ref{thm.dimension} follows. 
\printbibliography

@article{Sur04,
	Author = {Sury, B.},
	Journal = {Mathematics Magazine},
	Pages = {308--310},
	Title = {A parent of {B}inet's formula},
	Volume = {77},
	Year = {2004}}

@article{GriVas22,
	Author = {Grinberg, D. and Vassilieva, E.A.},
	Journal = {S{\'e}min. Loth. de Comb.},
	Title = {A q-Deformation of Enriched P-Partitions},
	Volume = {86B (FPSAC 2022)},
	Year = {2022}}

@article{GriVas21,
	Author = {Grinberg, D. and Vassilieva, E.A.},
	Journal = {S{\'e}min. Loth. de Comb.},
	Note = {arXiv:2202.04720v1},
	Title = {Weighted posets and the enriched monomial basis of QSym},
	Volume = {85B (FPSAC 2021)},
	Year = {2021}}

@article{Hof15,
	Author = {M. E. Hoffman},
	Journal = {Kyushu J. Math.},
	Note = {arXiv:math/0401319v3},
	Pages = {345--366},
	Title = {Quasi-symmetric functions and mod p multiple harmonic sums},
	Volume = {69},
	Year = {2015}}

@misc{GriRei20,
	Author = {Grinberg, D. and Reiner, V.},
	Note = {arXiv:1409.8356v7},
	Title = {{H}opf Algebras in Combinatorics},
	Url = {http://www.cip.ifi.lmu.de/~grinberg/algebra/HopfComb-sols.pdf},
	Year = {2020}}

@unpublished{Hsi07,
	Author = {Hsiao, S. K.},
	Title = {Structure of the peak Hopf algebra of quasisymmetric functions},
	Year = {2007}}

@article{Ste97,
	Author = {Stembridge, J.},
	Journal = {Trans. Amer. Math. Soc.},
	Number = {2},
	Pages = {763--788},
	Title = {Enriched $P$-partitions.},
	Volume = {349},
	Year = {1997}}

@book{Sta01,
	Author = {Stanley, R.},
	Publisher = {Cambridge University Press},
	Title = {Enumerative combinatorics},
	Volume = {2},
	Year = {2001}}

@article{Ges84,
	Author = {Gessel, I.},
	Journal = {Contemporary Mathematics},
	Pages = {289--317},
	Title = {Multipartite {P}-partitions and inner products of skew {S}chur functions},
	Volume = {34},
	Year = {1984}}

\end{document}